\documentclass[english,11pt]{amsart}
\usepackage[T1]{fontenc}
\usepackage{amsthm}
\usepackage{color}
\usepackage{euscript}
\usepackage{textcomp}
\usepackage{subfig}
\usepackage{amsmath}
\usepackage{amssymb}
\usepackage{amsthm}
\usepackage[english]{babel}
\usepackage{graphicx}
\theoremstyle{definition}
\newtheorem*{definition*}{Definition}
\newtheorem{theorem}{Theorem}
\newtheorem{lemma}[theorem]{Lemma}
\newtheorem{proposition}[theorem]{Proposition}

\newtheorem*{remark*}{Remark}
\newtheorem{remark}{Remark}
\newtheorem*{theorem*}{Theorem}

\newenvironment{proofof}[1]{\noindent {\bf{Proof of #1.}}}{ \hfill\qed\\ }

\def\e{\varepsilon}

\begin{document}
\title[Minimal interval exchange transformations with flips]{On the Hausdorff dimension of minimal interval exchange transformations with flips}
\author {Alexandra Skripchenko}
\address{Faculty of Mathematics, National Research University Higher School of Economics, Vavilova St. 7, 112312 Moscow, Russia}
\email{sashaskrip@gmail.com}

\def\curraddrname{{\itshape Mailing address}}
\author{Serge Troubetzkoy}
\address{Aix Marseille Universit\'e, CNRS, Centrale Marseille, I2M, UMR
  7373, 13453 Marseille, France}
\curraddr{ I2M, Luminy\\ Case 907\\ F-13288 Marseille CEDEX 9\\ France}
 \email{serge.troubetzkoy@univ-amu.fr}

\begin{abstract}
We prove linear upper and lower bounds for the Hausdorff dimension set of  minimal interval exchange transformations  with flips (in particular without periodic points), and a linear lower bound for the Hausdorff dimension of the set of non-uniquely ergodic minimal interval exchange transformations with flips.
 \end{abstract}
 \maketitle

\section{Introduction}
\subsection{The results}
We study interval exchange transformations  with flips (fIET). 
An fIET is an piecewise isometry of an interval to itself with a finite number of jump discontinuities, reversing  the orientation of at least one of the  intervals 
of continuity.
Interval exchange transformations with flips generalize the notion of orientation-preserving interval exchange transformations (IET). Each IET is the first return map to a transversal for a measured foliation on surface. In the same way, each fIET is the first return map to a transversal for a vector field on a non-orientable surface. More precisely one can associate each fIET with an oriented, but not co-orientable measured foliation on a non-orientable surface (see \cite{DySkr} for details).
fIETs also arise from oriented measured foliations on nonorientable surfaces \cite{DaNo}.
An additional motivation to study fIETs comes from so called billiards with flips (see \cite{No} for details).


The set of fIETs and of IETs on $n$ intervals are naturally parametrized by a subset of $\mathbb{R}^{n-1}$ and some discrete parameters. The ergodic properties of IETs are well known: almost all irreducible IETs are minimal (Keane \cite{Ke}) and almost all IETs are uniquely ergodic (Masur \cite{Ma}, Veech \cite{Ve2}).
Nogueira has shown that that fIETs  have completely different dynamics

\begin{theorem*}[\cite{No}]
Lebesgue almost every interval exchange transformation with flips has a periodic point.
\end{theorem*}

The measure in question is the Lebesgue measure on $\mathbb{R}^{n-1}$. In this article, we evaluate the Hausdorff dimension of the set $MF_n$ of minimal fIETs on $n$-interval, this set  is  subset of  fIETs on $n$-intervals without periodic points.  
We prove the following 
\begin{theorem}\label{minimal}
The Hausdorff dimension of the set $MF_n$ satisfies:
$$n-2\le Hdim (MF_n)<n-1.$$
\end{theorem}

We also prove a lower bound for the Hausdorff dimension of non-uniquely ergodic minimal fIETs. Let us denote this set by $NU\hspace{-0.03cm}E_n\subset MF_n$.

\begin{theorem}\label{NUE}
The  Hausdorff dimension of  the set of non-uniquely ergodic minimal $n$-fIETs satisfies:
$$\left[\frac{n-1}{2}\right]-1 \le Hdim({NU\hspace{-0.03cm}E}_n).$$
\end{theorem}

In the case $n=5$, using a result of  Athreya and Chaika \cite{AtCh} we have much better lower bound: 

\begin{proposition}\label{Jayadev}
For $n=5$
$$\frac{5}{2} \le Hdim(NU\hspace{-0.03cm}E_5).$$
\end{proposition}

The paper is organized as follows. 
 In Section \ref{Def} we give main definitions and briefly discuss known results and compare to the case of  interval exchange transformations
 without flips. 
In Section \ref{Rauzy} we describe our main tool -  Rauzy induction for interval exchange transformations with flips, and study its combinatorics. We also introduce the notion of the cocycle associated with the Rauzy induction.
 Section \ref{Markov} is dedicated to the Markov map associated with the Rauzy induction: we prove key features of this map and of the corresponding Markov partition, in particular we show that this map is uniform expanding in a sense of \cite{AvGoYo}.
 In Section \ref{Distortion} we prove some distortion estimations for the cocycle based on the so called Kerckhoff lemma (see \cite{Ker}). We mainly follow the approach suggested in \cite{AvGoYo} for interval exchange transformations and applied in \cite{AvRe} for linear involutions and in \cite{AvHuSkr} for systems of isometries.
 Section \ref{Roof} is about the roof function:  using the estimations from the previous section show that the roof function has exponential tails. The proof is also inspired by the similar result in \cite{AvGoYo}. 
In Section \ref{upper} we prove the upper bound for the Hausdorff dimension announced in Theorem \ref{minimal}. The proof is based on the argument presented in \cite{AvDe}: first, we show that the exponential tail of the roof function implies that the corresponding Markov map is fast decaying in a sense of \cite{AvDe} and then using \cite{AvDe} check that this property implies the estimation we are interested in.  
 Section \ref{lower} completes the proof of Theorem~\ref{minimal}: we show the lower bound applying a construction related
 to the one described in Nogueira in \cite{No}.  Using the same idea, we prove  Theorem \ref{NUE} and Proposition \ref{Jayadev}.
\subsection{Acknowledgements} 
We thank Pascal Hubert for very useful discussion and the anonymous referee for a careful reading of the manuscript and in particular for improving our lower bound.

AS was supported by RSF grant, project 14-21-00053 dated 11.08.14.

ST graciously acknowledge the support of R\'egion Provence-Alpes-C\^otes d'Azur; project APEX ``Syst\`emes dynamiques : Probabilit\'es et Approximation Diophantienne PAD''.
\section{Definitions and known results}\label{Def}
\subsection{Interval exchange transformations}
Let $I\subset \mathbb R$ be an interval (say, $I=[0,1)$) and $\{I_\alpha: \alpha \in \EuScript{A}\}$ be a partition of $I$ into subintervals, of the form $[a,b)$, indexed by some alphabet $\EuScript{A}$ on $n\ge 2$ symbols.  
An \emph{interval exchange transformation} (IET) is a bijective map from $I$ to $I$ which is a translation on each subinterval $I_{\alpha}$.

Such a map $f$ is determined by the following collection of combinatorial and metric data: 
\begin{itemize}
\item a vector $\lambda=(\lambda_\alpha)_{\alpha \in \EuScript{A}}$ of lengths of the subintervals. Each component $\lambda_\alpha$ is positive. 
\item a  permutation $\pi=(\pi_0,\pi_1)$ which is a pair of bijections $\pi_i: \EuScript{A}\rightarrow \{1,\cdots, n\}$ describing the ordering of the subintervals $I_\alpha$ before and after the map is iterated;
\end{itemize}

Each IET  preserves the measure and the orientation on $I$. 

\begin{definition*}\nonumber Let $f: X \to X$ be an invertible map. The \emph{orbit} of $x \in X$ is the subset 
$\{f^k(x)\;;\;k\in\mathbb Z\}$.
\end{definition*}

\begin{definition*}
A map $f:X \to X$ is called \emph{minimal} if all orbits are dense in $X$.
\end{definition*}

In \cite{Ke} M.\ Keane proved that almost all IETs for which  $\pi$ is irreducible are minimal. In this article we are only interested in transitive (minimal) IETs, so we assume that $\pi$ is always irreducible.
\begin{definition*}
A measure preserving map $f:(X,\mu) \to (X,\mu)$  is called \emph{uniquely ergodic} if it admits a unique invariant measure which is necessarily ergodic. 
\end{definition*}

H.\ Masur in \cite{Ma} and W.\ Veech in \cite{Ve2} proved that in case of irreducible permutations almost all IET are uniquely ergodic (and so, every invariant measure is a multiple of Lebesgue measure).

\subsection{Interval exchange transformations with flips}
fIETs are a generalization of IETs. Informally, an fIET $f$ is given by the partition of an interval $I$ into subintervals $I_{\alpha}$,
and a piecewise linear map from $\cup_{\alpha \in \mathcal{A}} \text{int} (I_{\alpha})$ into $I$ such that $f$ acts as an isometry on each $I_{\alpha}$, so that the images of interiors of partition elements  $I_{\alpha}$ do not overlap, and reverses the orientation of at least one subinterval  $I_{\alpha}$. 
Note that since the orientation of some intervals are reversed, defining $f$ on the left endpoints of the subintervals would not lead to a bijection of $I \to I$.
None the less if an fIET is minimal in the above sense, then if we take the closure of the symbolic model, then it is
minimal in the usual sense: every orbit is dense.

We proceed with the precise definition. Consider the interval $I=[0,1)$, its partition on subintervals $I_\alpha$ labeled by an alphabet
$\EuScript{A}$. Denote the length of the interval $I_\alpha$ by $\lambda_\alpha$ and consider the vector
 $\lambda=(\lambda_{\alpha})_{\alpha \in \EuScript{A}}$ (each component $\lambda_\alpha$ is positive).
Then  we define a pair of maps,  called a \emph{generalized permutation} $\hat\pi=(\pi_0,\pi_1),$ where $\pi_0: \EuScript{A}\mapsto \{1,\cdots,n\}$ is a bijection and $\pi_1: \EuScript{A}\mapsto \{-n,\cdots,-1,1,\cdots,n\}$ is a map with the absolute value $|\pi_1|$ of the map $\pi_1$ being a bijection. 
The pair of maps $\pi_0, |\pi_1|$ describe the order of the subintervals $I_\alpha$ before and after map is iterated. The
map  $\pi_1$ can be viewed as a signed permutation $\theta|\pi_1|$ where $\theta\in\{-1,1\}^{n}$. The signed permutation $\pi_1$ additionally
describes the set of flipped intervals which are marked by $\theta_\alpha=-1$. The set $F = \{\alpha\in  \EuScript A: \theta_\alpha=-1\}$ is called \emph{the flip set}.

\begin{definition*}
The fIET given by $(\hat \pi, \lambda) $ is the following map:
$$
f(x)=\begin{cases} x + w_\alpha & \mbox{if } x\in \text{int}(I_\alpha), \alpha \notin F, \\
w_\alpha-x, & \mbox{if } x\in \text{int}(I_\alpha), \alpha \in F, \end{cases}
$$
where $$w_\alpha = \sum_{|\pi_1|(\beta)<|\pi_1|(\alpha)}\lambda_\beta - \sum_{\pi_0(\beta)<\pi_0(\alpha)}\lambda_\beta,$$ if $\alpha\notin F$ and 
$$w_\alpha = \sum_{|\pi_1|(\beta)\le|\pi_1|(\alpha)}\lambda_\beta - \sum_{\pi_0(\beta)<\pi_0(\alpha)}\lambda_\beta$$ otherwise.
\end{definition*}

Throughout the article we will reserve the notation fIET to the case when the flip set is non-empty, and use the notation IET when the flip set is empty.

\begin{remark}\label{rem1}
IETs have a stronger property than minimality, they are forward and backwards minimal, i.e., every forward
orbit and every backwards orbit is dense.  Since some of orbits stop, this can not hold for fIETs, none the less our proof shows that this stronger property is almost true for fIETs:  every infinite forward orbit is dense, and the same for infinite backwards orbits.  
\end{remark}

As  mentioned above, the dynamics of fIETs  is diametrically opposite to the dynamics of classical IETs: the typical fIET has a periodic points (\cite{No}). The first example of a minimal fIET was constructed in \cite{No}. 
Some examples of minimal and uniquely ergodic fIETs  can be also found in \cite{GLMPZh,HL}, and a analyse of the number of periodic and minimal components of fIETs can be found in
\cite{NoPiTr}. 

\section{Rauzy induction}\label{Rauzy}
A detailed description of  Rauzy induction for fIETs can be found in \cite{No}, \cite{GLMPZh} and \cite{HL}. Here we present a brief scheme of the induction algorithm.

The term \emph{Rauzy induction} refers to  the operator  $R$ on the space of fIETs that associates to each fIET $f=(\lambda,\hat\pi)$ with irreducible $\hat \pi$ another fIET $f' = R(f) = R(\lambda,\hat \pi)$ which is the first return map induced by $f$ on a subinterval $[0,\nu]$ where $\nu$ is as follows:
$$
\nu=\begin{cases} \sum_{\alpha \in \EuScript{A}} \lambda_\alpha - \lambda_{\alpha_0}& \mbox{if }  \lambda_{\alpha_0}<\lambda_{\alpha_1}\\
\sum_{\alpha \in \EuScript{A}} \lambda_\alpha - \lambda_{\alpha_1}, & \mbox{if }  \lambda_{\alpha_0}>\lambda_{\alpha_1}. \end{cases}
$$
where $\alpha_0  := \pi^{-1}(n)$ and $\alpha_1 :=\pi_1^{-1}(\pm n)$.

Rauzy induction is not defined in the case that $ \lambda_{\alpha_0} = \lambda_{\alpha_1}$, the set of  fIETs satisfying this equality is of Hausdorff
dimension $n-1$ and thus does not interest us.

One can check that 
\begin{equation}\label{RI}
f'=(\lambda',\hat\pi')=\begin{cases} ((I_a(\hat\pi))^{-1}\lambda, a(\hat\pi))& \mbox{if }  \lambda_{\alpha_0}<\lambda_{\alpha_1}\\

((I_b(\hat\pi))^{-1}\lambda, b(\hat\pi)) & \mbox{if }  \lambda_{\alpha_0}>\lambda_{\alpha_1} \end{cases}
\end{equation}
where the transition matrices $I_a(\hat\pi), I_b(\hat\pi)\in SL(n,\mathbb Z)$ and transition maps $a(\hat\pi), b(\hat\pi)$ are defined below.

$$I_a(\hat\pi)=E+E_{\alpha_0,\alpha_1} \mbox{ and } I_b(\hat\pi)=E+E_{\alpha_1,\alpha_0}$$ where $E$ is an identity matrix and $E_{\alpha,\beta}$ is the elementary matrix containing $1$ as $(\alpha,\beta)$-th element.

Then the following four equations describe the  transition maps : 

$$
\mbox{if } \theta_{\alpha_0}:=1, \mbox{ then } a(\hat \pi)=\begin{cases}(\pi_0(\alpha), \pi_1(\alpha))& \mbox{if } |\pi_1|(\alpha)  \le \pi_1(\alpha_0) \\
(\pi_0(\alpha), \alpha_1) & \mbox{if } |\pi_1|(\alpha)=\pi_1 (\alpha_0) +1  \\
(\pi_0(\alpha),\pi_1(\alpha)+1) & \mbox{otherwise, } \end{cases} 
$$
$$
\mbox{if } \theta_{\alpha_1}=1, \mbox{ then }  b(\hat \pi):=\begin{cases}(\pi_0(\alpha),\pi_1(\alpha))& \mbox{if } \pi_0(\alpha) \le  \pi_0( \alpha_1)\\
(\alpha_0,\pi_1(\alpha)) & \mbox{if } \pi_0(\alpha)=\pi_0 (\alpha_1) + 1\\
(\pi_0(\alpha)+1,\pi_1(\alpha)) & \mbox{otherwise, } \end{cases}
$$
$$
\mbox{if } \theta_{\alpha_0}=-1, \mbox{ then }  a(\hat \pi):=\begin{cases}(\pi_0(\alpha),\pi_1(\alpha))& \mbox{if } |\pi_1|(\alpha) < \pi_1(\alpha_0) \\
(\pi_0(\alpha), \alpha_1) & \mbox{if } |\pi_1|(\alpha)=\alpha_0 \\
(\pi_0(\alpha),\pi_1(\alpha)+1) & \mbox{otherwise, } \end{cases} 
$$
$$
\mbox{if } \theta_{\alpha_1}=-1, \mbox{ then }  b(\hat \pi):=\begin{cases}(\pi_0(\alpha),\pi_1(\alpha))& \mbox{if } |\pi_1|(\alpha) < \pi_1(\alpha_0) \\
(\alpha_0,\pi_1(\alpha)) & \mbox{if } |\pi_1|(\alpha)=\alpha_0 \\
(\pi_0(\alpha)+1,\pi_1(\alpha)) & \mbox{otherwise}. \end{cases}
$$

See \cite{HL} for the more detailed description. 

In case that the operation $a$ was used, we say that the element $\alpha_0$ was the \emph{winner} and the element $|\alpha_1|$ was the \emph{loser}; 
in case that the operation $b$ was applied, the terminology is the opposite.

More generally, if the alphabet $\EuScript{A}$ has $n$ letters and  if $\alpha$ and $\beta$ were the last elements of non-signed permutation $\pi$, then, depending on the inequalities written above, we say that $\alpha$ is a winner and $\beta$ is a loser (or vice versa). 
Sometimes, since we also work with the signed permutation $\hat\pi,$ we can specify that $\alpha$ or $-\alpha$ was the winner or the loser. 
\begin{remark*} We will iterate the Rauzy induction map, however 
Rauzy induction is not defined everywhere, thus the iteration  stops if we arrive to a point outside its domain of definition  (see \cite{No} for an 
example when it stops). In case of an IET (without flips) which satisfies Keane's condition this never happens. 
\end{remark*}

\subsection{The Rauzy graph in the case of fIETs}
As in case of IETs, one can define Rauzy classes for fIETs. Given pairs $\hat\pi$ and $\hat \pi^{'}$, we say that $\hat\pi^{'}$ is a \emph{successor} of $\pi$ if there exist $\lambda, \lambda^{'}$ such that $R(\hat\pi,\lambda)=(\hat\pi^{'},\lambda^{'}).$ In the case of IETs, every permutation has exactly two successors; while for fIETs it has four successors.

However, the property that the successors of every irreducible permutations are also irreducible does not hold for fIETs. 
So, this relation defines a partial order on a set of irreducible permutations,  plus a  so called \emph{hole}, the set of all $(\hat \pi, \hat \lambda)$ without successors, the hole contains all
$(\hat \pi,\hat \lambda )$ with $\hat \pi$ reducible; 
this order can be represented by a direct graph $G$ that is called the \emph{Rauzy graph}.
As in case of IETs 
the connected components of this graph are called \emph{Rauzy diagram} or the \emph{Rauzy class}. Each Rauzy class $\EuScript{R}(\EuScript{A})$ contains a vertex that corresponds to the hole; there are no paths that start in the hole.

A \emph{path} of length $m\ge0$ in the diagram is a sequence (finite or infinite)
$v_0, \ldots, v_m,$ of vertices (signed permutations) and a sequence of arrows $a_1, \ldots a_m$ such that $a_i$ starts
at $v_{i-1}$ and ends in $v_i$. The following lemma holds:
\begin{lemma}
If  $\hat\pi$ and $\hat \pi^{'}$ are in the same Rauzy class and do not represent the hole then there exists an oriented
path in $G$ starting at $\hat\pi$ and ending at $\hat\pi^{'}$ that does not pass through the hole.
\end{lemma}
\begin{proof}
Since there are no arrows that depart from the hole vertex, the counting argument for IETs (see \cite{Vi}, Lemma 6.1) works in our case as well. 
\end{proof}

As it was mentioned above, the Rauzy induction data contains two parts: metrical (lengths of the intervals) and combinatorial (vertices of the Rauzy diagram).
The following  lemma holds:
\begin{lemma}
Each infinite path in the Rauzy diagram corresponds to a minimal fIET. 
\end{lemma}
\begin{proof}
See Theorem A in \cite{HL}. Infiniteness of the path means that the Rauzy induction is applicable infinite amount of times without finishing in the hole; it is exactly equivalent to the conditions that all the permutations obtained during the Rauzy induction application are irreducible and $\EuScript{O}(\EuScript{R})$ is infinite (see Section 3 of \cite{HL} for the definitions).
\end{proof}

In fact, every infinite path corresponds to a minimal fIET in the stronger sense mention in Remark \ref{rem1}.

For each Rauzy class $\EuScript R$ we denote the set of paths on it by $\Pi(\EuScript{R}$) (and refer to this set of paths as the Rauzy class);
for each path $\gamma$ (finite of infinite) there is the Rauzy operator $R_{\gamma}$ that corresponds to it. The matrix of the Rauzy induction, that is the product of $I_a$ and $I_b$, is also denoted by $R_{\gamma}$. We denote by the matrix $B_\gamma=R_\gamma^T$ sometimes will be referred as \emph{matrix of the cocycle} (the same construction is used in the case of IETs in connection with zippered rectangles as a suspension model, see \cite{Ve2} and \cite{Vi}). 

\begin{definition*}
Let $\EuScript{R}$ be a Rauzy class. A path $\gamma\in \Pi(\EuScript{R})$ is called
\emph{complete} if every letter $\alpha$ of the alphabet $\EuScript{A}$ is the winner of some arrow composing $\gamma$.
\end{definition*}
\begin{definition*}
We say that $\gamma\in \Pi(\EuScript{R})$ is \emph{positive} if all entries of the matrix corresponding to $R_\gamma$ are positive.
\end{definition*}

A more precise description of the positive paths in case of fIETs can be found in \cite{No} (Lemma 2.1).

\subsection{Acceleration}\label{sec.acceleration}
Rauzy induction has fixed neutral points which means that any absolutely continuous invariant measure is necessarily infinite.
 In the case of IETs  A.\  Zorich introduced an \emph{accelerated} algorithm of Rauzy induction (\cite{Zo}). It is an analogue of the acceleration of the Euclid algorithm.
This idea can be used directly in case of fIETs without any significant changes.

As was mentioned above, the matrix of Rauzy induction is a product of matrices $I_a$ and $I_b$ for different $a$ and $b$.
Let us for each stage $j$ of Rauzy induction define the indicator $\epsilon\in\{-1,1\}$ such that $\epsilon=1$ if one applied matrix $I_a$ and $\epsilon=-1$ if $I_b$ was the matrix of the corresponding stage of the Rauzy induction.
Then, \emph{Zorich induction} is given by the following operator:
$$(\lambda',\pi')=Z(\pi,\lambda)=R^n(\pi,\lambda),$$ where $n$ is the smallest $j \ge 1$ one such that $\epsilon^{(j)}=-\epsilon^{(0)}.$

One can associate Rauzy graphs with the accelerated induction in the same way as for the usual induction. In this paper we work with minimal fIETs. We exclude  the hole vertex from the  Rauzy graphs (we call this exclusion an ``adjustment'').

As in case of IETs, the following holds (see \cite{AvGoYo}):
\begin{lemma}
Every path that is long enough in terms of the accelerated Rauzy induction is complete and, moreover, positive.
\end{lemma}
\begin{proof}
An interval gets smaller at a step of the Rauzy induction if and only if the letter labeling this interval is a winner.
Proposition 10 in \cite{HL} implies that all the intervals are getting smaller if the path is long enough, thus 
 each letter in the alphabet must be a winner at least once.
 
On the other hand, every path which is a concatenation of at least $2n-3$ complete paths is positive. The proof repeats literally the one for IETs (see \cite{MaMoYo} for technical details). 
\end{proof}

\begin{remark*}
The statement about positivity of every path that is long enough can be also extracted from Lemma 2.1 in \cite{No}. There only minimal fIET were considered but the proof works in the same way even if fIET is not minimal but allows a sufficient amount of the iterations of the Rauzy induction.
\end{remark*}

\section{Markov map: definitions and properties}\label{Markov}
\subsection{Definition}
The Markov map $T$ is the projectivization of the induction map described above: 
$T(\hat\pi,\lambda) :=\big(\hat\pi' , \frac{\lambda'}{|\lambda|'}\big)=\Big (\hat\pi' , \frac{R^{-1}_{\gamma}\lambda}{||R^{-1}_{\gamma}\lambda||}\Big )$where $|\lambda| :=\beta_n(\lambda)=\sum_{\alpha \in \EuScript{A}} \lambda_\alpha$ and $\lambda', \hat\pi'$ are defined by the accelerated Rauzy induction (if the acceleration of Rauzy induction  is not defined at a point then we define its image by the last possible application of (\ref{RI})).

\subsection{Markov partition}
One can check that the map $T$ determines Markov partition $\Delta^{(l)}$ of the parameter space like in case of orientation-preserving IETs (see \cite{Vi}). 
The only difference is that the induction stops in some Markov cells; we are interested in the set of points where induction can be applied for an  infinite time. 
 
\subsection{Markov shift}
One can also consider the action of the non-accelerated Rauzy induction on the accelerated adjusted Rauzy graph. Then, each vertex of the adjusted Rauzy graph will split into countable number of vertices, and the same happens to the corresponding Markov cell. Then  the Rauzy induction corresponds to a Markov shift $\sigma$ in this coding on a countable alphabet. One can associate in a natural way a graph $\Gamma$ with such a Markov shift. $\Gamma$ can be obtained from the Rauzy graph by dividing every vertex into a countable number of vertices and adding a required arrows between these new vertices. 

\begin{definition*}
A countable Markov shift $\Theta$ with transition matrix $U$ and set of states $\EuScript{S}$ satisfies the  \emph{big images and pre-images property} (BIP) if there exist 
a finite subset of the states of the Markov shift such that the image under the action of the Markov shift of any state contains some element of this finite set, and furthermore the image of this finite subset contains the whole set of states. 
\end{definition*} 

As in case of IETs, the following  lemma holds for fIETs:
\begin{lemma}
The Markov shift $\sigma$ satisfies the BIP property.
\end{lemma}
\begin{proof}
One can check that it is enough to choose the finite subset of states such that each belongs to a different vertex of the accelerated Rauzy graph (for each Rauzy class);
note that the number of classes is finite and depends on $n$ (this determines  the parameter $m$ in the definition).  
\end{proof}

\subsection{The Markov map is uniformly expanding}
As  mentioned above, we use the so called Manhattan norm for vectors: 
$||v||=\sum_{i=1}^{n}v_i,$ note that all the vectors we work with are positive. We will also need the operator norm $|A|=\sup_{||v||=1}{Av}$ and the norm on continuous functions on compact sets: $||f||_{C^{0}(\Delta)}=\sup_{x\in\Delta}|f(x)|.$

\begin{definition*}
Let $L$ be a finite or countable set, let $\Delta$ be a parameter space, and let $\{\Delta^{(l)}\}_{(l\in L)}$ be a partition into open sets of a full measure subset of $\Delta.$
A map $Q: \cup_{l} \Delta^{(l)} \rightarrow \Delta$ is a \emph{uniformly expanding} map if:
\begin{itemize}
\item there exist a constant $k>1$ such that for each $l \in L$, $Q$ is a $C^{1}$ diffeomorphism between $\Delta^{(l)}$ and $\Delta$, and there exist a constant  $C_{(l)}$ such that 
$$k||v||\le||DQ(x)v||\le C_{(l)}||v||$$
for all $x\in \Delta^{(l)}$ and all tangent vectors $v\in T_{x}\Delta$ where  $DQ(x)$ denotes the derivative operator of $Q$ evaluation at $x$.
\item Let $J(x)$ be the inverse of the Jacobian of $Q$ with respect to Lebesgue measure. Denote by $\EuScript{H}$ the set of inverse branches of $Q$. The function $log J$ is $C^{1}$ on each set $\Delta^{(l)}$ and there exists $C>0$ such that, for all $h\in \EuScript{H}$,
$$||D((logJ)\circ h)||_{C^{0}(\Delta)}\le C.$$ 
\end{itemize}
\end{definition*}
We need also one more technical definition (see \cite{AvGoYo}):
\begin{definition*}
The positive path $\gamma$ on the Rauzy graph is called a \emph{neat} if starts and ends in the same vertex of the graph and the following condition holds: 
if $\gamma=\gamma_{s}\gamma_{0}=\gamma_{0}\gamma_{e}$ for some $\gamma_s$ and $\gamma_e,$ then either $\gamma=\gamma_0$ or $\gamma_0$ is trivial.
\end{definition*}

\begin{remark*}
The definition of a neat guarantees that the associated induction matrix is strictly positive.
\end{remark*}

We are mainly interested in minimal fIET, so we can assume that all paths we work with are long enough and so contain a neat. 
Taking a longer path is equivalent  to  accelerating the induction.  Zorich proposes to merge all the steps that correspond to the case when one iteration comprises all consequent steps where the winner is the same. Here we want to accelerate even more;  we 
would like that that all steps before the matrix will become positive to represent one iteration of the induction. We realize this induction by 
  taking a small subsimplex of the original parameter space and consider only the orbits that go back to this subsimplex.
If one prefers to work with the suspension model, we consider a special precompact section of the parameter space (see sections 4.1.2 - 4.1.3 in \cite{AvGoYo} for the details) and the first return map to this section. Naturally, in such a way orbits that never go back escape our control; however, the amount of these orbits among minimal orbits is negligible (see section \ref{correctness} for a precise statement). 

\begin{lemma}  The map 
$T$ is  uniformly expanding with respect to the Markov partition $(\Delta^{(l)})$.
\end{lemma}
\begin{proof}
The same lemma was proved for IETs in \cite{AvGoYo} (Lemma 4.3) and for linear involutions in \cite{AvRe} (Lemma 6.1). The main idea of the proof remains the same for our case and is briefly described below.

The first part of the definition of uniformly expanding for paths we work with follows from the positivity. More precisely, as mentioned above, we consider a long enough path $\gamma$ such that it contains some positive neat.
Then the induction matrix $R_{\gamma}^{-1}$ is a product of two matrices such that one of them is weakly contracting (is not expanding) and another one is strongly contracting (contraction coefficient is strongly larger than 1) with respect to Hilbert metric. The strong contraction comes from the positivity of the neat. 

More precisely, we do the following. Since the path we work with is positive, we can consider the inverse map of the map $Q$:
$$Q': \Delta\rightarrow\Delta^{l}$$ and prove that $Q'$ is uniformly contracting.
It is easy to check that $Q'\lambda' = \frac{C\lambda'}{|C\lambda'|},$ where $C$ is the matrix with strictly positive coefficients, relating the old and new coordinates (with respect to the Rauzy induction)  in the following way:


Then for each $\alpha \in \EuScript{A}$ we have
$$
   \sum_{\beta \in \EuScript{A}} c_{\alpha,\beta} \lambda'_{\beta}=\lambda_\alpha, 
$$

Now, using a result of Veech (see Lemma 3.1 in \cite{No}), we have  $J_{Q'}=\frac{1}{|C\lambda'|^{n}}=\frac{1}{\sum_{\alpha \in \EuScript{A}} C_\alpha \lambda'_{\alpha}},$
where $C_\alpha := \sum_{\beta \in \EuScript{A}} c_{\beta,\alpha}$ and $J$ is the  Jacobian (derivative). 
We have a renormalizing condition $\sum_{\alpha \in \EuScript{A}} \lambda'_\alpha =1$. 
Using this renormalizing condition, one can easily check that $\frac{1}{(\max_\alpha(C_\alpha))^{n}} \le J_Q' \le\frac{1}{(\min_\alpha(C_\alpha))^{n}}\le \frac{1}{n^n},$ the last inequality follows since $C$ is a positive integer valued matrix. 
The statement for $Q$ follows.

The second part of the definition of uniformly expanding can be verified as follows. First, we notice that an inverse
branch of the map $T$ can be written as $h(\hat\pi, \lambda)=(\hat \pi, \frac{R_{\gamma}\lambda}{||R_{\gamma}\lambda||})$.
Veech showed that
is $J\circ h=\frac{1}{||R_{\gamma}\lambda||^n}$ where $n$ is the number of
intervals of the IET \cite{Ve1}, Proposition 5.3; the proof holds verbatum for fIETs.
So we have:
$$\frac{J\circ h(\hat \pi, \lambda)}{J \circ h(\hat \pi, \lambda')}=\left(\frac{||R_{\gamma}\lambda||}{||R_{\gamma}\lambda'||}\right)^n \le \sup_{\alpha\in{1,\cdots,n}}\left(\frac{\lambda_{\alpha}}{\lambda'_{\alpha}}\right)^n \le e^{n\cdot dist(\lambda,\lambda')}.$$
and thus, $log J\circ h$ is Lipshitz with respect to the Hilbert metric (which is denoted by $dist$).
\end{proof}
\section{The distortion estimates}\label{Distortion}
\subsection{Conditional probabilities}
The distortion argument will  involve the study of the  forward images of the Lebesgue measure under the renormalization map. 
Following the strategy from \cite{AvGoYo} and \cite{AvRe}, we first construct a class of measures which is invariant as a class.

Let us consider the accelarated Rauzy graph and some path $\gamma$ in it. Let us fix the vertex $\hat\pi$ of this graph and the corresponding Rauzy class $\EuScript{R}$.

As in Section \ref{sec.acceleration}, $B_{\gamma}$ is the matrix of the cocycle corresponding to $\gamma$, and $B_\gamma^T$ denotes its transpose. 
Let $\mathbb{R}_+^n$ denote the positive cone and  $\langle \cdot, \cdot \rangle$ denote the  inner product, then define 
$$\Lambda_q := \{\lambda\in\mathbb R_+^{n}: \left\langle {\lambda,q} \right\rangle<1\},$$ and
$$\Delta'_{\gamma} := B^{T}_{\gamma}\mathbb R^{n}_{+}.$$

For $q=(q_1, \cdots, q_n) \in \mathbb R^{n}_{+}$ we define a measure $\nu_{q}$ on the $\sigma$-algebra $A\subset\mathbb R_+^{n}$ of Borel sets which are positively invariant (i.e. $\mathbb{R}_+ A = A)$:
$$\nu_{q}(A) := Leb(A \cap \Lambda_q).$$

Equivalently, $\nu_{q}$ can be considered as a measure on the projective space $\mathbb{R}P^{n-1}_{+}.$
Using Proposition 5.4 in \cite{Ve1}, one can check that 
$$\nu_{q}(\mathbb R^n_{+})=\frac{1}{n!q_1\cdots q_n},$$
$$\nu_{q}(\Delta'_\gamma)=\frac{1}{n!(B_{\gamma}q)_{1}\cdots (B_{\gamma}q)_{n}},$$
and 
$$\nu_{q}(B^{T}_{\gamma}A)=Leb(B^{T}_{\gamma}A\cap \Lambda_q)=Leb(
A\cap \Lambda_{B_{\gamma}q})=\nu_{B_{\gamma}q}(A).$$

The measures $\nu_q$ are used to calculate the probabilities of realization of different types of combinatorics related to the induction.

Let $\EuScript{R}$ be a Rauzy class and let $\gamma\in\Pi(\EuScript{R})$. Let $\Lambda_{q,\gamma} :=
\Lambda_{B_{\gamma}q}$. If $\Gamma\in\Pi(\EuScript{R})$ is a set of paths starting with the same $\hat \pi\in\EuScript{R}$, let $\Lambda_{q,\Gamma}=\cup_{\gamma\in\Gamma}\Lambda_{q,\gamma}$. We define
$$P_q(\Gamma|\gamma) := \frac{Leb(\Lambda_{q,\Gamma_{\gamma}})}{Leb(\Lambda_{q,\gamma})}=\frac{\nu_{q}(\cup_{\gamma'\in\Gamma_{\gamma}}\Delta'_{\gamma'})}{\nu_{q}(\Delta'_{\gamma})},$$ 
where $\Gamma_{\gamma}\subset\Gamma$ is the set of paths starting by $\gamma$.

Suppose $\gamma$ is a path of length one (an arrow),  let $\hat \pi\in \EuScript{R}$ be the permutation from which $\gamma$ starts and denote by $\alpha\in \EuScript{A}$ the winner and by $\beta\in \EuScript{A}$ the loser of the first iteration of the Rauzy induction (without acceleration). 
Then, the conditional probability related to the given combinatorics is defined by
$$P_{q}(\gamma|\hat \pi) :=\frac{q_{\beta}}{(q_{\alpha}+q_{\beta})}.$$

For a long path, for $\EuScript{A'}\subset \EuScript{A}$ and $q\in \mathbb R^{\EuScript{A}}_{+}=\mathbb R^{n}_{+},$ let $N_{\EuScript{A'}}(q) :=\prod_{\alpha\in \EuScript{A'}}q_{\alpha}$  and, let $N(q) :=N_{\EuScript{A}}(q)$, and define $$P_{q}(\gamma|\hat \pi) :=\frac {N(q)}{N(B_{\gamma}q)}.$$ 

Let us introduce a partial order on the set of paths: for two given path $\gamma, \gamma^{'}$ we say that $\gamma\le\gamma^{'}$ if $\gamma^{'}=\gamma\gamma_e$ for some $\gamma_{e}$. We say that the subset $\Gamma$ is \emph{disjoint} if no two elements are comparable with respect to this order. 
Now, for every family $\Gamma\subset\Pi(\EuScript{R})$  such that any $\gamma\in\Gamma$ start by some element $\gamma_s\in\Gamma_s,$ for every $\hat \pi\in\EuScript{R}$ we define
\begin{equation}\label{defPq}
P_q(\Gamma|\hat\pi) := \sum_{\gamma_s\in\Gamma_s}P_q(\Gamma|\gamma_s)P_q(\gamma_s|\hat\pi).
\end{equation}
Note also that 
\begin{equation}\label{Gamma|pi}
P_q(\Gamma|\hat\pi)\le P_q(\Gamma_s| \hat \pi)\sup_{\gamma_s\in\Gamma_s}P_q(\Gamma|\gamma_s).
\end{equation}

\subsection{Kerckhoff lemma}
In this section we prove the key estimate necessary for the estimation of the distortion properties of the cocycle matrix. 
The idea that was used for the first time in \cite{Ker} for IETs is the following: in order to control how the induction distorts a vector which was originally well balanced, one has to check that the ratio between the norms of rows (or equivalently, columns) of the matrix of the cocycle (equivalently, of the induction matrix) can rarely be very high.
More formally, we have the following: 
\begin{lemma}\label{Ker}
Let $\hat\pi\in\EuScript R$ be irreducible. 
For any $T>1, q \in \mathbb{R}^{\EuScript{A}}_{+}, \alpha \in \EuScript{A}$

$$P_{q}\Big ( \big \{ \gamma \in \Gamma_{\alpha}(\hat \pi): (B_{\gamma}q)_{\alpha}>Tq_{\alpha} \big \} \big | \hat \pi \Big )<\frac{n}{T},
$$
where $\Gamma_{\alpha}(\hat \pi)$ denotes the set of paths starting at $\hat \pi$ with no winner equal to $\alpha$ and $n$ is the
number of intervals of the fIET.
\end{lemma}

 \begin{remark*}
A very close statement for IETs was proven in \cite{Ker}[Proposition 1.3]. In \cite{No}[Proposition 3.5] one can find a bit weaker version of the statement we are interested in.
\end{remark*}

\begin{proof}

We follow the strategy from \cite{AvGoYo}. 
Consider $\Gamma^{(k)}_{\alpha}(\hat \pi)$  the set of paths of length at most $k$
in $ \Gamma_{\alpha}(\hat \pi)$.
We will prove the inequality 
\begin{equation}\label{dd}
P_{q}\Big ( \big \{ \gamma \in \Gamma_{\alpha}^{(k)}(\hat\pi): (B_{\gamma}q)_{\alpha}>Tq_{\alpha} \big \} \big | \hat \pi \Big )< \frac{n}{T},
\end{equation}
by induction on $k$. The Theorem then follows immediately.

We start with the case $k = 1$. Consider a path $\gamma$ of length $1$, i.e.\ a simple arrow. 
Apriori there are two cases depending of the permutation $\hat \pi$.
The first case is if neither $\alpha$  nor $-\alpha$ are the last element of the either line of the permutation  $\hat \pi$, then 
$(B_\gamma q)_\alpha = q_\alpha$.
Thus  the set which we considered in Equation \eqref{dd} is  the empty set and thus the inequality
holds trivially.

On the other hand $\alpha$ (or $-\alpha$) can be the last element, but in this case it most be the loser since 
by the definition of $\Gamma_{\alpha}(\hat\pi)$ it could not be a winner. Denote the winner by $\beta$.
Then  $P_q(\gamma|\hat\pi)=\frac{q_\alpha}{q_{\alpha}+q_{\beta}}$. Thus
since $(B_\gamma q)_\alpha = q_\alpha + q_\beta$, the inequality  $(B_{\gamma}q)_{\alpha}>Tq_{\alpha}$ is equivalent to  $q_\alpha/(q_\alpha+q_\beta)< 1/T$. Let $\beta_i$ be the collection of all such
possible winners, and $\gamma_i$ denote the corresponding path of length 1, i.e.
$\Gamma_\alpha^{(1)}(\hat \pi) = \{ \gamma_1, \dots , \gamma_\ell\}$. Then since the paths are all disjoint
we have
\begin{align*}
P_{q}\Big ( \big \{ \gamma \in \Gamma_{\alpha}^{(1)}(\hat\pi): (B_{\gamma}q)_{\alpha}>Tq_{\alpha} \big \} \big | \hat \pi \Big ) & =  \sum_{(i: q_\alpha + q_{\beta_i} > T q_\alpha)} \frac{q_\alpha}{q_\alpha + q_{\beta_i}}\\
&<  \mathit{card} \big (i: q_\alpha + q_{\beta_i}  >T q_\alpha \big ) \cdot  \frac1T\\
& \le   \frac{n}{T}.
\end{align*}
This finishes the proof of the case $k=1$.

 In the same way, the case $k$ follows immediately from the case $k - 1$ when none of the rows of $\hat\pi$ end with $\alpha$ (or $-\alpha$).  If this is not the case then we can assume that the top row of $\hat\pi$ ends with $\alpha$ and the
bottom row with $\beta$. Then every path $\gamma\in \Gamma^{(n)}_{\alpha}(\hat\pi)$ starts with the bottom arrow
starting at $\hat \pi$ (i.e.\ $\beta$ wins because we are in $\Gamma_{\alpha}(\hat\pi)$).
Let us denote the path of length 1 consisting of  this arrow by $\gamma_s$ and let $q' = B_{\gamma_s}q$. We have $q'_\alpha=q_\alpha+q_\beta$ and $P_q(\gamma_s|\hat\pi) = \frac{q_{\alpha}}{q'_{\alpha}}$.
The inequality follows by the induction hypothesis  and Equation \eqref{Gamma|pi}.
\end{proof}
Now we apply the Kerckhoff lemma to obtain some more subtle estimations on the distortion.
The main idea is as follows: we want to define the first return time to the small subsimplex of the original parameter space; 
however, some (minimal) orbits that start in the fixed subsimplex will not go back since they stick somewhere close to the boundary; 
our purpose is to check that the probability of this event is in some sense low.
 
We mainly follow the strategy suggested in Appendix A of \cite{AvGoYo}.
Before we actually state the theorem, let us introduce some useful notation:
$$\EuScript{A'} \subset \EuScript{A};$$
$$m_{\EuScript{A'}}(q) = \min_{\alpha\in\EuScript{A'}}q_{\alpha};$$
$$m(q)=m_{\EuScript{A}}(q)$$
$$m_{k}(q) = \max_{\{\EuScript{A'} \subset \EuScript{A}: |\EuScript{A'}|=k\}}m_{\EuScript{A'}}(q);$$
$$M_{\EuScript{A'}}(q)=\max_{\alpha\in\EuScript{A'}}q_{\alpha};$$
$$M(q) = \max_{\alpha\in\EuScript{A}}q_{\alpha}.$$

The principal result of this section is the following 
\begin{theorem}\label{A2}
Let $\hat\pi\in\EuScript R$ be irreducible.
There exists $C>1$ such that for all $q\in \mathbb{R}^{\EuScript{A}}_{+}$ 
$$P_{q}\Big ( \big \{\gamma: M(B_{\gamma}q)<C \min\{m(B_{\gamma}q),M(q)\} \big \} | \hat\pi \Big )>C^{-1}.$$
\end{theorem}

\begin{proof}
The main idea of the proof comes from \cite{AvGoYo}: one should consider all the subsets $\EuScript{A'} \subset \EuScript{A}$ of fixed cardinality $k$ and prove that for each $1\le k\le n$ there exists $C >1$ (depending on $k$) such that
\begin{equation}\label{with_k}
P_{q}\Big ( \big \{\gamma: M(B_{\gamma}q)<C\min\{m_k(B_{\gamma}q),M(q)\} \big \} | \hat\pi \Big )>C^{-1}.
\end{equation}
The case $k=n$ implies the desired statement.
The proof is by induction on $k$. 

For $k=1$ we have $M(B_{\gamma}q)=m_1(B_{\gamma}q)\geq M(q)$ and so one has to estimate the probability of the following event $E_1=\{\gamma: M(B_{\gamma}q)<CM(q)\}$. Let  $E^c$ denote the complement of the event $E$; then $$E_{1}^c=\{\gamma: M(B_{\gamma}q)\ge CM(q)\}=\bigcap_{i}\{\gamma: M(B_{\gamma}q)\ge Cq_i\}, i=1,\cdots,n.$$ 
But $M(B_{\gamma}q)=(B_{\gamma}q)_{j}$ for some $j$, so $E_{1}^c\subset \cup_{j=1}^n E_{j}$ where $E_{j}=\{\gamma: (B_{\gamma}q)_{j}\ge Cq_j\}$.
Then Lemma \ref{Ker} implies that for any $C > 0$ we have $P_q(E_{1}^c)<\sum_{j} P_q(E_{j})<n^2C^{-1}$, and thus $P_q(E_{1})\ge 1-\frac{n}{C}>\frac{1}{C}$ for any $C>n^2+1$.

Now, we make the induction step. Let us assume that Equation \eqref{with_k} holds for some $1\le k<n$ with some constant $C_0$. We denote by $\Gamma$ the set of shortest possible paths $\gamma$ starting at $\hat \pi$ with $M(B_{\gamma}q)<C_0\min\{m_k(B_{\gamma}q), M(q)\}$ (by the length of the path we mean the number of steps of the non-accelerated induction). We will construct a family of paths such that the statement holds for all paths from this family for some subset $\EuScript{A''}$ of cardinality $k+1$ (it is enough because $m_k$ is a maximum taken over all subsets of fixed cardinality).

Recall that $$m_{k}(B_{\gamma}q) = \max_{\{\EuScript{A'} \subset \EuScript{A}: |\EuScript{A'}|=k\}}m_{\EuScript{A'}}(B_{\gamma}q)=m_{\EuScript{A'}}(B_{\gamma}q)$$ for some $\EuScript{A'}$ with cardinality $k$; so 
there exists a non-empty set of paths $\Gamma_1\subset \Gamma$ such that 
if $\gamma\in \Gamma_1$ then $m_k(B_{\gamma}q)=m_{\EuScript{A'}}(B_{\gamma}q)$. 
Since non-empty sets of finite paths have positive probability we can find $C_1 > 1$ such that $P_q(\Gamma_1|\hat \pi)>C_1^{-1}$.

For each $\gamma_s\in \Gamma_1$ consider all paths of  $\gamma :=\gamma_s\gamma_e$ with minimal length such that $\gamma$ ends at a permutation $\hat\pi_e$, such that the top or the bottom row of $\hat\pi_e$ (possibly both) ends with some element that does not belong to $\EuScript{A'}\cup-\EuScript{A'}$. 
Let $\Gamma_2$ be the collection of paths $\gamma$ obtained in this way. So, there exists $\tilde C_2>0$ such that $P_q(\Gamma_2|\hat \pi)>\tilde C_2^{-1}$. 
By the induction assumption and definition of $\Gamma$ 
\begin{equation}\label{distortion}
M(B_{\gamma_s}q)\le C_0 \min({m_k(B_{\gamma_s}q), M(q)}),
\end{equation}
Since induction does not  decrease the norm, this implies that 
 all components of the vector $B_{\gamma_s}q$ belong to the interval $[m(q), C_{0}M(q)]$.
We work only with paths of minimal length, thus the winner must change on each step. This implies that 
the total length of  $\gamma_e$ can not exceed $2|\EuScript{A'}|\le 2k$, and that on each step the loser is different. 
Each step in the induction corresponds to the following action on matrix of the cocycle: some component of the vector $B_{\gamma_s}q$ is added to another one. The fact that the loser is always different means that we always add different components to each other. On each step the norm of the loser component  is added to the norm of the winner component, so every time we add not more than $C_{0}M(q)$, due to \eqref{distortion}. 
The total number  of steps is bounded by $2n$. So, there exists a constant $C'_2 = (n+1)C_{0}M(q)>0$ such that  $$M(B_{\gamma}q)<C'_{2}M(B_{\gamma_s}q).$$

Then, for $C_2=\max(\tilde C_2, C'_2)$ we have $P_q(\Gamma_2|\pi)>C_2^{-1}$  and for $\gamma\in\Gamma_2$ $$M(B_{\gamma}q)<C_{2}M(B_{\gamma_s}q).$$

Let $\Gamma_3$ be a set of paths $\gamma := \gamma_s\gamma_e$, where $\gamma_s\in \Gamma_2, (B_{\gamma}q)_\alpha\le 2n^2 \cdot (B_{\gamma_s}q)_\alpha$ for all $\alpha\in \EuScript{A'}$ and also several combinatorial conditions hold: 
\begin{itemize}
\item the winner of the last arrow of $\gamma_e$ belongs either to $\EuScript{A'}$ or to $-\EuScript{A'}$;
\item the winners of other arrows does not belong to $\EuScript{A'}$ and to $-\EuScript{A'}$.
\end{itemize}

This set of paths is essentially the set described in Lemma \ref{Ker} (the winners of arrows should not belong to $\EuScript{A'}$) with exactly one difference: one has to consider a conditional probability $P_q(\Gamma_3|\gamma_s)$ with respect to $\gamma_s$, not some fixed permutation $\hat \pi$. More precisely, for fixed $\gamma_s$ let $\Gamma_3^{\gamma_s} := \{\gamma_e: \gamma_s \gamma_e \in \Gamma_3  \}$. 
Suppose that $\gamma_s$ ends at the permutation $\hat \pi$, then $P_q(\Gamma_3|\gamma_s)=P_{B_{\gamma_s}q}(\Gamma^{\gamma_s}|\hat \pi)$.
So, we apply Lemma \ref{Ker} with $T=\frac{1}{2n}$ and see that $P_q((\Gamma_3|\gamma_s)^{c})\le\frac{1}{2n}$ and so 
$P_q(\Gamma_3|\gamma_s)\ge 1-\frac{1}{2n}=\frac{2n-1}{2n}>\frac{1}{2}$.

This implies that $P_q(\Gamma_3|\hat \pi)=P_q(\Gamma_3|\gamma_s)P_q(\gamma_s|\hat \pi)>(2C_2)^{-1}.$

Let $\gamma :=\gamma_s\gamma_e\in \Gamma_3$. 
If $M(B_{\gamma}q)>2n \cdot M(B_{\gamma_s}q)$ we consider  $\gamma' = \gamma_s \gamma_e$ 
the minimal length prefix  of $\gamma$  for which the same inequality holds: $M(B_{\gamma'}q)>2n \cdot M(B_{\gamma_s}q)$. Then, there exists $\alpha \not \in \pm \EuScript{A'}$ such that $M(B_{\gamma'}q)=(B_{\gamma'}q)_{\alpha}\le 4n \cdot M(B_{\gamma_s}q)$ (it follows from the fact that $\gamma_1$ is of minimal length and that the $Mq$ can at most double after one step of the non-accelerated Rauzy induction).
Together with the assumption that the statement holds for $k$ it implies that $$m_{\EuScript{A'}}(B_{\gamma_1}q)>(C_{0}C_{2}4n)^{-1}M(B_{\gamma_1}q).$$

If $M(B_{\gamma}q)\le 2nM(B_{\gamma_s}q)$ the loser $\alpha$ of the last arrow of $\gamma$ satisfies the following inequality:
$$(B_{\gamma_1}q)_{\alpha}\ge (C_{0}C_{2}4n)^{-1}M(B_{\gamma_1}q)$$ (by the same calculation as above). 
Our $\alpha$ did not belong to $\EuScript{A'}$.

In any case, we construct the family $\Gamma_4$ (it contains $\gamma'$ presented above) and $\EuScript {A'}$ of cardinality $k+1$ for which the Inequality (\ref{with_k}) holds. 
\end{proof}

Now we prove some more subtle distortion estimate.

\begin{theorem} \label{thm:proba-estimate}
For every $\hat{\gamma}\in\Pi(\EuScript{R})$ there exist $\delta>0, C>0$ such that for every $\hat \pi\in\EuScript{R}$, $q\in \mathbb R^{\EuScript{A}}_{+}$ and for every $T>1$

$P_{q}\big (\gamma$ cannot be written as $\gamma_{s}\hat{\gamma}\gamma_{e}$ and $M(B_{\gamma}q )>TM(q)| \hat \pi  \big )\le CT^{-\delta}.$
\end{theorem}

\begin{remark*}
The restriction on the paths means that we only consider paths that do not contain $\hat{\gamma}$ as a proper part.
\end{remark*}

The most important point of the argument we use is that the estimates that we prove in Theorem \ref{thm:proba-estimate} are uniform with respect to $q$.

The proof of the theorem is based on the following two lemmas that can be considered as the corollaries of Theorem \ref{A2}.

\begin{lemma} \label{condproba}
There exists $C'>1$ such that for any permutation $\hat \pi$
$$P_{q} \big ( \{ \gamma : 
M(B_{\gamma}q)>C'M(q), m(B_{\gamma}q)<M(q)  \} |\hat \pi \big )<1-\frac{1}{C'}.$$
\end{lemma}

\begin{proof}
From Theorem \ref{A2} we know the lower bound of the probability of the following event:
\begin{equation}\label{C}\nonumber
P_q (X \cup Y| \hat \pi )>\frac{1}{C},
\end{equation}
where $X := X_{1} \cap X_{2} $ is defined by
$$
\begin{aligned} 
X_{1} & : = \{ \gamma: M(B_{\gamma}q)<Cm(B_{\gamma}q) \}\\ 
X_{2} & := \{ \gamma: m(B_{\gamma}q) < M(q)\} \\ 
\end{aligned} 
$$
and $Y := Y_{1} \cap Y_{2} $ is given by
$$
\begin{aligned} 
Y_{1} & := \{ \gamma : M(B_{\gamma}q) < CM(q)\}\\ 
Y_{2} & := \{ \gamma: M(q) \le m(B_{\gamma}q)\}. \\ 
\end{aligned} 
$$
Suppose $x \in X_{1}^c \cap X_{2}$, then $x \not \in X_{1}$ and thus $x \not \in X$.
Furthermore $x \not \in Y_{2}$ since $X_{2} \cap Y_{2} = \emptyset$, and thus $x \not \in Y$.
Thus
$$P_q(X_{1}^c \cap X_{2}| \hat \pi) < 1 - \frac1C.$$
Then the lemma follows for any $C' > C$ since
$$ \{\gamma : M(B_{\gamma}q)>C'M(q), m(B_{\gamma}q)<M(q)  \}  \subset  X_{1}^c \cap X_{2}.$$
\end{proof}

We write $\hat \gamma \sqsubset \gamma$ if there exist
non-empty $\gamma_s$ and $\gamma_e$ such that
 $\gamma = \gamma_s \hat \gamma \gamma_e$, otherwise we write $\hat \gamma \not \sqsubset \gamma$.
The next lemma follows from the previous one and is also important for the proof of Theorem \ref{thm:proba-estimate}.
\begin{lemma}\label{PL}
For any $\hat{\gamma}\in \Pi(\EuScript{R})$ there exist $M\ge 0, \rho<1$ such that for any $\hat \pi\in \EuScript{R}, q\in \mathbb{R}^{\EuScript{A}}_{+}$ 

$P_{q}(\gamma: \hat \gamma \not \sqsubset \gamma$ and $M(B_{\gamma}q)>2^{M}M(q)|\hat \pi) \le \rho.$
\end{lemma}

\begin{proof}
Fix $M_{0}$ large enough (we will choose it precisely later) and let $M :=2M_{0}$. We consider the set 
$$\Gamma := \{\gamma: \gamma \text { of minimal length satisfying } \hat \gamma \not \sqsubset \gamma \text{ and } M(B_{\gamma}q)>2^{M}M(q)\}.$$

As  mentioned above, $M(B_{\gamma}q)$ can not increase more than twice for a path of the length one.
So any path in $\Gamma$ can be written as $\gamma=\gamma_{1}\gamma_{2}$ where $\gamma_{1}$ is the shortest path such that 
$$M(B_{\gamma_{1}}q)>2^{M_{0}}M(q),$$
and neither  $\gamma_1$ nor $\gamma_2$ coincide with $\gamma$.

Let us denote the set of such $\gamma_{1}$ by $\Gamma_{1}$.
It  follows directly from minimality that $\Gamma_{1}$ is disjoint in terms of \cite{AvGoYo} which means that any path is not a part of some other path from the same set. 

Now  consider the subset $\tilde \Gamma_{1}$ of  $\Gamma_{1}$ consisting of all $\gamma_{1}$ such that 
$m(B_{\gamma_{1}}q)\ge M(q)$ (or, equivalently, $M_{\EuScript{A'}}(B_{\gamma_{1}}q)\ge M(q),$ for all non-empty $\EuScript{A}'$).

By Lemma \ref{condproba} choosing $M_{0}$ large enough we have that 
$$P_{q}(\Gamma_{1}\setminus \tilde\Gamma_{1}|\hat \pi)<1-\frac{1}{C}$$
with some constant $C>1.$

Now we use the strategy from \cite{AvRe}.  We fix some permutation $\hat\pi_{e}$ and consider the shortest path $\gamma_{\hat\pi_{e}}$ starting at $\hat\pi_{e}$ and containing $\hat{\gamma}$ (if there are several such paths choose one). We define $\gamma_s$ by  $\gamma_{\hat \pi_{e}}=\gamma_{s}\hat{\gamma}$.
Then, if $M_{0}$ is large enough, we can assume that 
\begin{equation}
|B_{\gamma_{\hat \pi_{e}}}|<2^{M_{0}-1}.
\label{B}
\end{equation}
If $\hat \pi_{e}$ is the end of some $\gamma_{1}\in\Gamma_{1}$, then $P_q(\Gamma|\gamma_1)=P_{B_{\gamma_1}q}(\Gamma^{\gamma_1}|\hat \pi_e),$
where $\Gamma^{\gamma_1} := \{\gamma_e: \gamma = \gamma_1 \gamma_e \in \Gamma\}.$ 
So, since $\gamma$ does not contain $\hat\gamma$ as a proper part it follows that
\begin{equation}\label{proper}
P_{q}(\Gamma|\gamma_{1})\le 1 - P_{B_{\gamma_{1}}q}(\gamma_{\hat \pi_{e}}|\hat \pi_{e})
\end{equation}
because $\Gamma^{\gamma_1}$ and the set of $\gamma_{\hat\pi_e}$ do not intersect and together they fill not more than the set of possible continuations of $\gamma_1$. 

If $\gamma_{1}\in \tilde\Gamma_{1}$, $P_{q}(\Gamma|\gamma_{1})$ can be estimated directly in terms of the measures of subsimplices of the original simplex: if $N(q) = q_{1}\cdots q_{n},$ then
$$P_{B_{\gamma_{1}}q}(\gamma_{\hat \pi_{e}}|\hat \pi_{e}) = \frac {N(B_{\gamma_{1}}q)}{N(B_{\gamma_{\hat \pi_e}}B_{{\gamma_1}}q)}.$$
We need to make several estimates, first all by the definition of  $\tilde \Gamma_1$  we have $m(B_{\gamma_1}q) \ge M(q)$ and thus
$N(B_{\gamma_{1}}q)\ge (M(q))^n$.
Next Inequality (\ref{B}) implies that
$$N(B_{\gamma_{\hat \pi_e}}B_{{\gamma_1}}q)< (2^{M_{0}-1}2^{M}M(q))^{n},$$
because it follows from the definition of $\Gamma_1$ that $M(B_{\gamma_1}q)<2^{M}M_{q}.$
So, it is easy to see now that 
\begin{equation}\label{low}
P_{B_{\gamma_{1}}q}(\gamma_{\hat \pi_{e}}|\hat \pi_{e})\ge 2^{-3nM_{0}}.
\end{equation}

We start with the case $P_{q}(\tilde\Gamma_{1}|\hat\pi)\ge \frac{1}{2C}$.
Starting with the definition of $P_q$ (Equation \eqref{defPq}) we have
$$\begin{aligned}
P_{q}(\Gamma|\hat\pi)&=\sum_{\gamma_1\in\Gamma_1}P_q(\Gamma|\gamma_1)P_q(\gamma_1|\hat\pi)\\&=\sum_{\gamma_1\in\tilde\Gamma_1}P_q(\Gamma|\gamma_1)P_q(\gamma_1|\hat\pi)+\sum_{\gamma_1\in\Gamma_1\setminus\tilde\Gamma_1}P_q(\Gamma|\gamma_1)P_q(\gamma_1|\hat\pi)\\
 &\le \big  (\!\! \sup_{\gamma_1\in\tilde\Gamma_1}  P_q(\Gamma|\gamma_1) \big )\cdot \sum_{\gamma_1 \in  \tilde \Gamma_1}  P_q(\gamma_1|\hat\pi)+\big (\!\!\!\!\!\!  \sup_{\gamma_1\in\Gamma_1\setminus\tilde\Gamma_1} \!\!\!\! P_q(\Gamma|\gamma_1) \big ) \cdot \!\!\!\!\!\! \sum_{\gamma_1 \in  \tilde \Gamma_1 \setminus \tilde \Gamma_1}\!\!\!\! P_q(\gamma_1|\hat\pi)\\
 & = \big (\!\! \sup_{\gamma_1\in\tilde\Gamma_1}P_q(\Gamma|\gamma_1) \big )  \cdot P_q(\tilde\Gamma_1|\hat\pi)+\big ( \!\!\!\!\!\! \sup_{\gamma_1\in\Gamma_1\setminus\tilde\Gamma_1}\!\!\!\! P_q(\Gamma|\gamma_1) \big ) \cdot P_q(\Gamma_1\setminus\tilde\Gamma_1|\hat\pi). \\
\end{aligned}$$
In the last line we used the fact that $\Gamma_1$  is disjoint.

Inequalities \eqref{proper}, \eqref{low} imply that $ \sup_{\gamma_1\in\tilde\Gamma_1}P_q(\Gamma|\gamma_1)  < 1 - 2^{-3nM_0}$, using this 
and the facts that $\sup_{\gamma_1\in\Gamma_1\setminus\tilde\Gamma_1} P_q(\Gamma|\gamma_1)  \le 1$ and
$P_q(\Gamma_1\setminus\tilde\Gamma_1|\hat\pi) \le 1 - P_q(\tilde \Gamma_1 | \hat \pi)$ along with the assumption $P_{q}(\tilde\Gamma_{1}|\hat\pi)\ge \frac{1}{2C}$ yields
$$\begin{aligned}
P_{q}(\Gamma|\hat\pi) & \le P_{q}(\tilde\Gamma_{1}|\hat\pi)\cdot(1-2^{-3nM_{0}})+1-P_{q}(\tilde\Gamma_{1}|\hat\pi)\\
& =1-P_{q}(\tilde\Gamma_{1})\cdot2^{-3nM_{0}}\\
& \le 1-\frac{2^{-3nM_{0}}}{2C}.
\end{aligned}$$

Now consider the case $P_{q}(\tilde\Gamma_{1}|\hat\pi)<\frac{1}{2C},$
then by Inequality \eqref{Gamma|pi} $P_{q}(\Gamma_{1}|\hat\pi)<1-\frac{1}{C}+\frac{1}{2C}=1-\frac{1}{2C}.$
So, Lemma \ref{PL} holds with $\rho = 1 - \frac{2^{-3nM_{0}}}{2C}$.
\end{proof}

Now we turn to the proof of the theorem.

\begin{proofof}{Theorem \ref{thm:proba-estimate}}
Let $M$ and $\rho$ be as in the Lemma \ref{PL}. For a given $T>1$ let $k$ be the maximal integer such that such that $T\ge 2^{k(M+1)}$. Let   $\gamma$ be the shortest path that does not include $\hat{\gamma}$ as a proper part and such that $M(B_{\gamma}q)>2^{k(M+1)}M(q)$. Then $\gamma$ can be written in the following way: $\gamma=\gamma_{1}\cdots\gamma_{i}\cdots\gamma_{k},$ where for each $i$ $\gamma_{(i)}=\gamma_1\cdots\gamma_i$ is the shortest path such that 
\begin{equation}\label{path}
M(B_{\gamma_{i}}q)>2^{i(M+1)}M(q).
\end{equation}
Then all such $\gamma_{(i)}$ comprise a set $\Gamma_{(i)}$ for each $i$, and these sets are disjoint because each path has to be the shortest one satisfying (\ref{path}). 
Now, Lemma \ref{PL} and (\ref{path}) imply that for all $\gamma_{(i)}\in \Gamma_{(i)}$
$$P_{q}(\Gamma_{(i+1)}|\gamma_{(i)})\le \rho.$$
So $P_{q}(\Gamma|\hat \pi)<{\rho}^{k}$. The result follows from the definition of $k$.
\end{proofof}

\section{The roof function}\label{Roof}
\subsection{Definition}
The construction of the roof function that we present in this section is based on the idea of renormalization provided by Veech in \cite{Ve2}.
Fix some positive complete path $\gamma_{*}$ starting and ending at the same permutation $\hat \pi$, and the subsimplex of the parameter space that corresponds to this path $\Delta_{\gamma_{*}}$. 
We are interested in the first return map to the subsimplex $\Delta_{\gamma_{*}}$. The connected components of the domain of this first return map are given by the $\Delta_{\gamma\gamma_{*}}$ where $\gamma$ is a path that contains $\gamma_{*}$ as a part, but does not start with $\gamma_{*}\gamma_{*}.$ Thus  the first return map $T$ restricted to such a component satisfies
\begin{equation}\label{T}
T(\lambda, \hat \pi)=\Big(\frac{R^{-1}_{\gamma}\lambda}{||R^{-1}_{\gamma}\lambda||},\hat \pi\Big),
\end{equation}
where $R$ is the matrix of the Rauzy induction.

\begin{definition*}
The {\em roof function} is the return time to the connected component described above:
$$r(\lambda, \hat \pi) := -\log||R^{-1}_{\gamma} \lambda||.$$
\end{definition*}

\begin{remark*}
As it was mentioned in \cite{AvGoYo} and \cite{AvRe}, with such a definition one works with the precompact sections because the path $\gamma_{*}$ is positive.
\end{remark*}

\begin{remark*}
In the case that the point belongs to the hole, we define the {\em roof function} in a natural way as the logarithm of the possible number of iterations of the Rauzy induction. 
Equivalently, one can consider $\gamma$ as the longest possible path up to the hole and apply the standard definition. 
\end{remark*}
\subsection{Correctness of the model}\label{correctness}
In this subsection we partly follow the strategy from \cite{AvHuSkr}.

In the suspension model we work with, the orbits that do not come back to the fixed precompact section are not considered.  We need to show
that they do not contribute to the Hausdorff dimension of the fractal we are studying. 
First of all, one can see that the Rauzy graph splits into several connected components (depending on combinatorics of the original permutation or, equivalently, on the geometry of the foliation). We will work with a particular connected components, prove our statement for each of them and then conclude the final result by sum over all the components. 
We start by noticing that 
 the following properties of the Markov map hold:
\begin{enumerate}
\item the BIP property implies that each small simplex of the Markov partition (let's say that it $\Delta_{\gamma}$ where $\gamma$ is a corresponding complete path in the Rauzy graph) is mapped on the whole parameter space $X$, and the map is surjective;
\item the Markov map $T$ is uniformly expanding and so the Jacobian of the map from $\Delta_\gamma$ to $X$ is bounded.
\end{enumerate}

Let us recall that for each Rauzy class the subset of the parameter space that gave rise to minimal interval exchange transformations with flips $MF_n$ has a fractal structure for the following reason: the point belongs to $MF_n$ iff the Rauzy induction can be applied infinitely many times to the corresponding nIET and never arrives to the hole. 
So, now we denote by $MF_n(\Delta_{\gamma})=\Delta_{\gamma}\cap MF_n$.  

The map $T$ is surjective since it is the projectivization of the induction map.
The BIP property together with the fact that $T$ is a surjective uniformly expanding map imply that $Hdim(MF_n(\Delta_{\gamma_*}))=Hdim (MF_n)$  (see also \cite{AvHuSkr} where the same statement was proved for minimal systems of isometries).  
The same argument can be used for $\Delta_{\gamma}$ and $\Delta_{\gamma'},$ where $\gamma'=\gamma\hat\gamma\gamma$ for some suitable $\hat\gamma$. Therefore, the orbits that escape the control do not contribute to the Hausdorff dimension of the fractal we study, and our suspension model is correct. 

\subsection{Exponential tails}
In this section we prove that the roof function constructed above has \emph{exponential tails}. We follow the strategy from \cite{AvGoYo}.\begin{definition*}
A function $f$ has \emph{exponential tails} if there exists $\sigma > 0$ such that $\int_{\Delta}e^{\sigma f}dLeb<\infty.$
\end{definition*}
\begin{theorem}
The roof function $r$ defined above has exponential tails.
\end{theorem}
\begin{proof}
This theorem is a direct corollary of Theorem \ref{thm:proba-estimate}.
The main idea is the same as was used in the case of IETs (see \cite{AvGoYo}): $-\log||(B^{T}_{\gamma_{*}})^{-1} \lambda||$ is the ``Teichm\"uller'' time needed to renormalize the support interval to unit length. Then time is divided into pieces of exponential size. For each piece, we apply Theorem \ref{thm:proba-estimate}. 

Indeed, in the previous section we constructed the set of Lebesgue measures $\nu_{q}$ on $\Lambda_{q}$ that depended on vector $q$. Let us consider $q_{0}=(1,\dots,1)$ and the corresponding measure $\nu_{q_{0}}$. 
Let us recall that our parameter space for a given Rauzy class $\EuScript{R}$ can be viewed as $\mathbb R^{\EuScript{A}}_{+}\times{\EuScript{R}}$ with the renormalization condition $\sum_{i=1}^{n}\lambda_i=1$. In particular for a given permutation we define $\Delta_{\hat \pi}=\mathbb R^{\EuScript{A}}_{+}\times\hat \pi$ with the same renormalization condition. Fix $\hat \pi$ and consider the natural projection of $\Lambda_{q}$ to the set $\Delta_{\hat \pi}$.
The pushforward $\nu$ of the measure $\nu_{q_0}$ under this projection is a smooth function on the parameter space of the Markov map $T(\lambda,\hat \pi)$ (see \cite{AvRe} or \cite{AvGoYo}). Thus, in order to prove the theorem, it is enough to show that 
\begin{equation}\label{nu}
\nu\{x\in \Delta_{\gamma_{*}}:r(x)\ge log T\}\le CT^{-\delta}
\end{equation}
for some $C$ and some $\delta.$

The connected component of the domain of the Markov map $T(\lambda,\hat \pi)$ that intersects the set $W=\{x:\{x\in \Delta_{\gamma_{*}}:r(x)\ge log T\}\le CT^{-\delta}\}$ is of the form  $\Delta_{\gamma}$ for some $\gamma$.
such that $\gamma$ can not be a concatenation of more than three copies of $\gamma_{*}$ and $$M(B_{\gamma}q_{0})>C^{-1}T,$$ for some constant $C$ that depends on $\gamma_{*}$. This first requirement on $\gamma$ follows from the fact that
 we work with the first return maps while if $\gamma$ is a concatenation of four copies of $\gamma_{*}$, one can take $\gamma_{*}\gamma_{*}$ as a path of the first return (and all other properties will be the same); 
the second statement follows from the definition of the roof function and the definition of the set $W$. 

Now we estimate the measure of the interesting set in terms of probabilities of corresponding events:
$\nu\{x\in \delta_{\gamma_{*}}:r(x)\ge log T\}\le P_{q_{0}}(\gamma$ does not contain some $\hat\gamma$ as a proper set and $M(B_{\gamma}q_{0})>C^{-1}T|\hat \pi)<CT^{-\delta}.$
The statement of the theorem follows now from Theorem \ref{thm:proba-estimate}.
\end{proof}

\section{The upper bound proof}\label{upper}
\subsection{Fast decaying Markov maps}
Let $\Delta$ be a measurable space and $T: \Delta \rightarrow \Delta$ be a Markov map. We will denote the corresponding Markov partition by $\Delta_{(l)}, l\in \mathbb {Z}$. 
 \begin{definition*}
We say that $T$ is \emph{fast decaying} if there exists $C_{1}>0, \alpha_{1}>0$ such that 
\begin{equation}\label{fade}
\sum_{\mu(\Delta^{(l)})\le\varepsilon}\mu(\Delta^{(l)})\le C_{1}\varepsilon^{\alpha_{1}}
\end{equation}
for all $0<\varepsilon<1.$
\end{definition*}

\begin{lemma}
Exponential tails of the roof function implies fast decaying property of the Markov map.
\end{lemma}
\begin{proof}
First, one can check directly that $JT(\lambda,\hat \pi)=e^{n  \cdot r(\lambda,\hat \pi)}$ (it follows from the formula proved by Veech (\cite{Ve1}, see also Lemma 3.1 in \cite{No}). 
We claim that the lemma follows from this formula and the fact that the measure of subsimplices (Markov cylinders) are proportional to $|DT|$.
The scheme of the proof is as follows: the measure of a subsimplex is proportional to the inverse of the Jacobian, thus one begins by replacing   the measures of subsimplices in the sum of Equation \eqref{fade} by the corresponding jacobians; using the above formula the jacobians are then replaced by the exponential of the roof function; so we only need to evaluate the following sum: 
$$\sum_{a\, : \, r(a)\ge N}e^{-nr(a)},$$
where by $a$ we denoted a point of $\Delta^{(l)}$ since the roof function is locally constant. 
The last sum can be evaluated using the exponential tails of the roof function (namely, the convergence of the corresponding integral): 
first, the exponential tail implies that $Card (Y(N))\leq Ce^{(n-\sigma)N},$ where $Y(N)$ is the set of partition subsets for which $r(a)$ is between $N$ and $N+1$ (see \cite{AvHuSkr2}, Lemma 17); then the sum we are interested in it can be estimated from above by a geometric series with ratio $e^{-\sigma}$.
\end{proof}

\section{Lower bounds}\label{lower}
In this section we show the lower bound in Theorem \ref{minimal} and prove Theorem \ref{NUE} and Proposition \ref{Jayadev}.

\subsection{Construction}\label{constr}
Let $S: [0,1) \to [0,1)$ be an  $(n-1)$-IET. 
We construct an $n$-fIET $T$  from $S$ as follows. Consider $\alpha_0$ so that $\pi_1(\alpha_0) =1$. Then define $T : [0,1+\beta_{i_0} - \beta_{i_0-1})$ as follows:
$$
T(x)=\begin{cases} 
Sx, & \mbox{if } x \in \text{int}(I_\alpha) \text{ for any } \alpha \ne \alpha_0, \\
1 - x + \beta_{i_0}, & \mbox{if } x \in \text{int}(I_{\alpha_0}),\\
1 - x + \beta_{i_0} - \beta_{i_0-1}, & \mbox{if } x > 1. \\
 \end{cases}
$$

Notice that  the first return map of $T$ to the interval ${[0,1)}$ is the map $S$ except at the end points of the intervals $I_\alpha$
where $T$, being an fIET is not defined. Suppose now that $S$ is minimal, then it immediately follows 
that every bi-infinite $T$-orbit is dense.  But since $S$ is an IET it has a stronger property: 
 the forward and backwards  $S$-orbit of each point is dense $[0,1)$. Thus even if the forward  $T$ orbit of a point $x$ arrives at the end point of some interval $I_\alpha$, the backwards $T$ orbit will be dense (and vice versa).
Since all but countably many IETs (without flips) are minimal the lower bound of Theorem \ref{minimal} follows.

\subsection{Non-uniquely ergodic case.}
The idea is the same as above but we fix $S$ with the maximal possible number of invariant ergodic measures (we denote this number by $k$).
Let us denote these ergodic measures by $\mu_1,\cdots, \mu_k$. Consider a probability vector
$\vec{\e} = (\e_1,\dots,\e_k)$  and  the corresponding measure $\mu_{\vec{\e}}=\sum \e_i\mu_i$.
Using Lemma 1 from \cite{Ka}, we have the following:
\begin{lemma}
There exists an m-IET $S_{\vec{\e}}$ such that $(S,\mu_{\vec{\e}})$ is metrically isomorphic to $(S_{\vec{\e}}, Leb)$.
\end{lemma}

We need to to check is that the map $\vec{\e} \to S_{\vec{\e}}$ is almost surely
invertible.  
\begin{lemma}
For a set of full measure of  $\vec{\e}$  the length vectors  $ (\lambda^{\vec{\e}}_1, \dots, \lambda^{\vec{\e}}_{n-1})$ are distinct.
\end{lemma}
\begin{proof}
Assume that  for two vectors $\vec{\e} \ne \vec{\e'}$ the lengths coincide. Let $\mu^i$ denote the 
$\mu_i$ measure  of the interval $[0, \sum_{\alpha \in \EuScript{A}}\lambda_\alpha]$.
Then the assumption implies that for each $i=1,\dots,n-1$
$$\e_1\mu^i_1+\cdots+\e_k\mu^i_k=\e'_1\mu^i_1+\cdots+\e'_k\mu^i_k$$  and thus 
\begin{equation}
\sum_{j=1}^{k}(\e_j-\e'_j)\mu^i_j=0. \label{codim}
\end{equation}
But $$\sum_{j=1}^{k}\e_j=\sum_{j=1}^{k}\e'_j$$ and so Equation (\ref{codim}) is equivalent to 
\begin{equation}
\sum_{j=1}^{k-1}(\e_j-\e'_j)C_{j}=0,
\end{equation}
where $C_j :=\mu^i_j-\mu^i_{n-1}$. Since the $C_j$ are constants ($S$ is fixed)  the set of $\vec{\e}$ for which (\ref{codim}) holds is a subspace of codimension $1$ in the parameter space, and therefore the statement of the lemma holds. 
\end{proof}
Now, one can apply the construction of Section \ref{constr} to get a family of non-uniquely ergodic fIETs. 
The Hausdorff dimension of this subset is not smaller than $k-1$ where $k$ is the number of invariant measures of the IET $S' := S_{\vec{\e}}$ constructed above.
$S'$ is $(n-1)$-IET. It was proven by Sataev in \cite{Sa} that $k=g$ where $g$ is the genus of a translation surface associated with IET (see, for example, \cite{Vi}). 
Therefore, $2g = n-1-r+1$ where $r$ is the number of singularities of the translation surface; on the other hand, $r$ can be estimated using the Euler characteristics of the surface and, in particular, the minimal value of $r$ is $1$ (and this value is always obtained). So it implies that $2g \le n-1$ and so $$Hdim (NU\hspace{-0.03cm}E_n)\ge \left[\frac{n-1}{2}\right]-1.$$ 
 \begin{remark*}
If one is interested in lower bound of $Hdim(NU\hspace{-0.03cm}E_n)$ for a particular combinatorics determined by the Rauzy class of IET $S'$, it is easy to see 
that for any integer $i$ between $\left[\frac{n+2}{4}\right]$ and $\left[\frac{n-1}{2}\right]$  one can find a Rauzy class of IET $S'$ constructed above such that the lower bound of the Hausdorff dimension of non-uniquely ergodic minimal fIETs is $i-1$. It follows from the fact that $r$ can take any integer value between $1$ and $2g-2$.
\end{remark*}

\subsection{Non-uniquely ergodic fIETs on 6 intervals}
One can also combine the construction presented above with the following result by J. Athreya and J. Chaika:\\


\begin{theorem*}[\cite{AtCh}]
The Hausdorff dimension of the set of non-uniquely ergodic $4$-IET on $[0,1)$ with the permutation $\pi_0=(4,3,2,1)$ is~$\frac{5}{2}$.
\end{theorem*}

Proposition \ref{Jayadev} follows since the   construction increases the number of intervals by  1


\begin{thebibliography}{99}

\bibitem[AtCh]{AtCh}
\newblock J.\ Athreya, J.\ Chaika,
\newblock \emph{The Hausdorff Dimension of Non-Uniquely Ergodic directions in $\mathcal{H}(2)$ is almost everywhere $1/2$};
\newblock \\ http://arxiv.org/abs/1404.4657.

\bibitem[AvDe]{AvDe}
\newblock A.\ Avila, V.\ Delecroix 
\newblock \emph {Weak mixing direction in non-arithmetric Veech surfaces},
\newblock arXiv: 1304.3318v1.


\bibitem[AvRe]{AvRe}
\newblock A.\ Avila and M.\ J.\ Resende, 
\newblock \emph {Exponential mixing for the Teichm\"uller flow in the space of quadratic differentials}, 
\newblock Comm. Math. Helv. \textbf{87} (2012), 589--638.


\bibitem[AvGoYo]{AvGoYo} 
\newblock A.\ Avila, S.\ Gou\"ezel and J.-C.\ Yoccoz, 
\newblock \emph {Exponential mixing for Teichm\"uller flow}, 
\newblock Publ. Math. IH\'ES 104 (2006), 143--211. 

\bibitem[AvHuSkr]{AvHuSkr}
\newblock A.\ Avila, P.\ Hubert and A.\ Skripchenko,
\newblock \emph {On the Hausdorff dimension of the Rauzy gasket}, 
\newblock http://arxiv.org/abs/1311.5361. 

\bibitem[AvHuSkr2]{AvHuSkr2}
\newblock A.\ Avila, P.\ Hubert and A.\ Skripchenko,
\newblock\emph{Diffusion for chaotic plane sections of 3-periodic surfaces},
\newblock http://arxiv.org/abs/1412.7913.

\bibitem[DaNo]{DaNo}
\newblock C.\ Danthony and A.\ Nogueira,
\newblock\emph{Measured foliations on nonorientable surfaces}
Annales scientifiques de l'Йcole Normale Supйrieure (1990)
\textbf{23}, 3(1990),  469--494.


\bibitem[DySkr]{DySkr}
I.\ Dynnikov, A. Skripchenko,
\newblock\emph{Minimality of interval exchange transformations with restrictions},
http://arxiv.org/abs/1510.03707

\bibitem[GLMPZh]{GLMPZh}
\newblock C.\ Gutierrez, S.\ Lloyd, V.\ Mervedev, B.\ Pires and E.\ Zhuzhoma,
\emph{Transitive circle interval exchange maps with flips}, 
Disc. Cont. Dyn. Syst. \textbf{26}, 1 (2010), 251--263.

\bibitem[HL]{HL}
C. Angosto Hern\'andez and G. Solar Lopez, 
\emph{Minimality and the Rauzy-Veech algorithm for interval exchange transformations with flips},
Dynamical Systems, \textbf{28}, 4 (2013), 539--550.

\bibitem[Ka]{Ka}
\newblock A.\ Katok,
\newblock \emph{Interval exchange transformations and some special flows are not mixing},
\newblock Israel J. Math. \textbf {35:4}(1980), 301--310.

\bibitem[Ke]{Ke}
\newblock M.\,Keane.
\newblock \emph{Interval exchange transformations},
\newblock Math. Z. \textbf{141} (1975), 25--31.

\bibitem[Ker]{Ker}
\newblock S.\ Kerckhoff,
\newblock \emph {Simplicial systems for interval exchange maps and measured foliations},
\newblock Ergod. Th. \& Dynam. Sys. 5 (1985), 257 -- 271.

\bibitem[Ma]{Ma}
H.\ Masur, 
\emph{Interval exchange transformations and measured foliations},
Ann.\ Math. (2) \textbf{115}, 1(1982), 169--200. 

\bibitem[MeNo]{MeNo}
\newblock R.\ Meester and Th.\ Nowicki,
\newblock \emph {Infinite clusters and critical values in two-dimensional circle percolation},
\newblock Isr. J. Math. 68:1 (1989), 63-81.

\bibitem[MaMoYo]{MaMoYo}
S. Marmi, P. Moussa and J.-C. Yoccoz, 
\emph{The cohomological equation or Roth type interval exchange
transformations},
J. Amer. Math. Soc. 18 (2005), 823--872.


\bibitem[No]{No}
\newblock A.\ Nogueira, 
\newblock \emph{Almost all interval exchange transformations with flips are nonergodic},
\newblock Ergod. Th. \& Dynam. Sys. 9(1989), 515--525.

\bibitem[NoPiTr]{NoPiTr} 
\newblock A.\ Nogueira, B.\ Pires, and, S.\ Troubetzkoy,
\newblock \emph{Orbit structure of interval exchange transformations with flips}, 
\newblock Nonlinearity 26 (2013) 525--537.

\bibitem[Sa]{Sa}
E.\ Sataev, 
\emph{On the number of invariant measures for flows on orientable surfaces},
Mathematics of the USSR-Izvestiya \textbf{9:4} (1975), 813--830.

\bibitem[Ve1]{Ve1} W.\ Veech,
\emph{Interval exchange transformations}, J.\ Analyse Math. 33 (1978), 222--272.

\bibitem[Ve2]{Ve2}
 W.\ Veech,
\emph{Gauss measures for transformations on the space of interval exchange maps},
Ann.\ Math. (2) \textbf{115}, 1 (1982),  201--242.

\bibitem[Vi]{Vi}
M.\ Viana,
\emph{Ergodic Theory of Interval Exchange Maps},
Revista Mathematica Complutense \textbf {19:1} (2006), 7--100.

\bibitem[Yo]{Yo}
J. -C. Yoccoz, 
\emph{Continued fraction algorithms for interval exchange maps: an introduction},
\'Ecole de Physique des Houches (2006), Frontiers in Number Theory, Physics and Geometry, volume 1: On random matrices, zeta functions and dynamical systems, 2006.

\bibitem[Zo]{Zo}
\newblock A.\ Zorich,
\newblock \emph {How do the leaves of a closed 1-form wind around a surface},
\newblock Transl. of the AMS, Ser.2, vol. 197, AMS, Providence, RI (1999), 135--178.


\end{thebibliography}
\end{document}